\documentclass[MSNbibl,number,citesort,seceqn,dvips]{arxbj}
\usepackage{upgreek}

\aid{0}
\volume{20}
\issue{1}
\pubyear{2014}
\firstpage{265}
\lastpage{281}
\doi{10.3150/12-BEJ485} 

\makeatletter

\newcommand{\llbracket}{[\![}
\newcommand{\rrbracket}{]\!]}

\newcommand{\Law}{\operatorname{Law}}

\newtheorem{theorem}{Theorem}[section]
\newtheorem{prp}{Proposition}[section]

\newtheorem{cor}{Corollary}[section]

\renewcommand{\i}{\mathrm{i}}
\newcommand{\ex}{\mathrm{e}}
\newcommand{\di}{\mathrm{d}}

\newcommand{\g}{\gamma}

\newcommand{\bI}{\mathbb{I}}
\newcommand{\bR}{\mathbb{R}}
\newcommand{\E}{\mathbf{E}}
\renewcommand{\P}{\mathbf{P}}
\newcommand{\cF}{\mathcal{F}}

\makeatother

\begin{document}
\begin{frontmatter}

\title{Small noise asymptotics and first passage times of integrated
Ornstein--Uhlenbeck processes driven by $\alpha$-stable L\'evy processes}
\runtitle{Integrated Ornstein--Uhlenbeck process}

\begin{aug}
\author{\fnms{Robert} \snm{Hintze}\thanksref{e1}\ead[label=e1,mark]{robert.hintze@uni-jena.de}} \and
\author{\fnms{Ilya} \snm{Pavlyukevich}\corref{}\thanksref{e2}\ead[label=e2,mark]{ilya.pavlyukevich@uni-jena.de}}
\runauthor{R. Hintze and I. Pavlyukevich} 
\address{Friedrich Schiller Universit\"at Jena, Fakult\"at f\"ur
Mathematik und Informatik, Institut f\"ur Stochastik,
Ernst Abbe Platz 2, 07743 Jena, Germany.
\printead{e1};\\ \printead*{e2}}
\end{aug}

\received{\smonth{5} \syear{2012}}
\revised{\smonth{11} \syear{2012}}

%
\begin{abstract}
In this paper, we study the asymptotic behaviour of one-dimensional
integrated Ornstein--Uhlenbeck \mbox{processes} driven by $\alpha$-stable
L\'evy processes of small amplitude. We prove that the integrated
Ornstein--Uhlenbeck process converges weakly to the underlying
$\alpha$-stable L\'evy process in the Skorokhod $M_1$-topology which
secures the weak convergence of first passage times. This result
follows from a more general result about approximations of an arbitrary
L\'evy process by continuous integrated Ornstein--Uhlenbeck processes
in the $M_1$-topology.
\end{abstract}

%


\end{frontmatter}

\section{Introduction}

Consider a dimensionless Langevin equation for the motion of a particle
with a position $x^\varepsilon$
subject to a linear friction force $F=-A \dot x^\varepsilon$, $A>0$ (Stokes'
law for friction force) and a random noise
$\dot l$ of a small amplitude $\varepsilon>0$
%
\begin{equation}
\label{eqlang} \ddot x^\varepsilon=-A \dot x^\varepsilon+\varepsilon\dot
l.
\end{equation}
Denoting by $v^\varepsilon:=\dot x^\varepsilon$ the velocity process,
we understand this
equation as a two-dimensional
equation in a phase space $(v,x)$ which can be written in the integral
form as
%
\begin{eqnarray}
\label{eqv} v_t^\varepsilon&=&v_0-A \int
_0^t v_s^\varepsilon\,\di s+
\varepsilon l_t,
\\
\label{eqx} x_t^\varepsilon&=&x_0+\int
_0^t v_s^\varepsilon\,\di s, t
\geq0,\qquad x_0,v_0\in \mathbb{R}.
\end{eqnarray}
The study of the dynamics of $x^\varepsilon$ and $v^\varepsilon$ in
the Gaussian case,
that is, when $l=b$ is a standard Brownian motion,
has a long history.
In this case, the velocity process $v^\varepsilon$ is a
Ornstein--Uhlenbeck (OU) process, and the displacement process
$x^\varepsilon$,
being the integrated Ornstein--Uhlenbeck process,
is often referred to as Langevin's Brownian motion. For example, it is
well known that for zero initial conditions
$x_0=v_0=0$, strong friction and
large amplitude, $A=\varepsilon\to+\infty$, the displacement process
$x^\varepsilon$ can be considered as a good physical approximation of
the Brownian motion (see, e.g., Chapter 2 in Horsthemke and
Lefever \cite{HorLef84}).

The dynamics of the integrated OU processes driven by non-Gaussian
L\'evy processes attracted attention recently in financial mathematics
in the context of stochastic volatility models, see Barndorff-Nielsen
\cite{Bar98}, and Barndorff-Nielsen and Shephard \cite{BarShe03}.
Garbaczewski and Olkiewicz \cite{GarOlk00} studied integrated OU
processes driven by a 1-stable (Cauchy) process. Al-Talibi, Hilbert and
Kolokoltsov \cite{altalibi10} established convergence in probability of
marginals of an integrated OU process driven by an $\alpha$-stable
L\'evy process in the limit of large friction parameter. Chechkin,
Gonchar and Szyd{\l}owski \cite{ChechkinGS02} studied the equation
(\ref{eqlang}) (with $\varepsilon=1$, in two- and three-dimensional
setting) in a model of plasma in an external constant magnetic field
and subject to an $\alpha$-stable L\'evy electric forcing.

Our present research is mainly motivated by the this paper and focuses
on the first passage times of
the displacement process $x^\varepsilon$ in the limit of small
amplitude $\varepsilon\to0$ under the assumption that the driving
process $l$
is a non-Gaussian $\alpha$-stable L\'evy process.
We refer the reader to the works by
Lefebvre \cite{Lef89} and Hesse \cite{Hesse91} where first passage
problems for integrated Ornstein--Uhlenbeck processes driven by
Brownian motion
were studied.

Let us briefly describe the outline of the paper. First, we shall show
that on a certain $\varepsilon$-dependent time
scale the integrated OU process $A x^\varepsilon$ weakly converges to the
driving process $l$ in the sense of finite-dimensional
distributions. Further, we shall establish a
stronger convergence of the processes in an appropriate topology. We
notice that since the
driving $\alpha$-stable L\'evy process has c\`adl\`ag paths and the
integrated Ornstein--Uhlenbeck process
is absolutely continuous, no convergence in the uniform topology or in
the Skorokhod $J_1$-topology is possible.
Thus, we prove the convergence in the weaker Skorokhod $M_1$-topology
which is still strong enough to
ensure the continuity of the running supremum or the inverse function
of a process, and to guarantee the convergence
of the first passage times. As a by-product, we obtain
an approximation result for an arbitrary L\'evy process by absolutely
continuous integrated OU processes in the
$M_1$-topology.

\section{Object of study and main result}

Let $(\Omega, \cF,\mathbb F,\P)$ be a filtered probability space
satisfying the usual conditions. On this probability
space, consider a L\'evy process $l$ with c\`adl\`ag paths and a
L\'evy--Khintchine representation $\E\ex^{\i u l_t}=\ex^{t \Psi
(u)}$, where
%
\begin{equation}
\label{eqPhi} \Psi(u)=-\frac{\sigma^2}{2}u^2+\i\mu u+\int
_{\bR\setminus\{0\}} \bigl(\ex^{\i u y} -1-\i u y\bI_{\{|y|\leq1\}}
\bigr) \nu(\di y) ),\qquad u\in\bR,
\end{equation}
with $\sigma\geq0$, $\mu\in\bR$, and a jump measure $\nu$ satisfying
the conditions
$\nu(\{0\})=0$ and $\int_{\bR\setminus\{0\}}(y^2\wedge1) \nu
(\di
y)<\infty$. In particular, we shall be interested in
non-Gaussian strictly $\alpha$-stable L\'evy processes $l^{(\alpha
)}=(l^{(\alpha)}_t)_{t\geq0}$, $\alpha\in(0,2)$, for
which the closed-form representation of the
characteristic exponent $\Psi$ is known to be equal to
\[
\Psi(u)= \cases{\displaystyle  -c|u|^\alpha \biggl( 1-\i\beta\operatorname
{sign}(u)\tan \frac{\uppi \alpha}{2} \biggr), &\quad $\alpha\in(0,1)\cup(1,2)$,
\vspace*{2pt}\cr
\displaystyle -c |u|
\biggl( 1+\i\beta\frac{2}{\uppi } \operatorname {sign}(u)\ln u \biggr), &\quad $
\alpha= 1, u\in\bR$, }
\]
$c>0$ and $\beta\in[-1,1]$ being a scale and skewness parameters (see,
e.g., Theorem 14.15 in Sato \cite{Sato-99}).
The well-known self-similarity property of $l^{(\alpha)}$ will be used
in the following:
$\Law(\varepsilon l^{(\alpha)}_{t/\varepsilon^\alpha}, t\geq
0)=\Law(l^{(\alpha)}_t,
t\geq0)$ for any $\varepsilon>0$.

For any $A>0$, $\varepsilon>0$, any $v_0$, $x_0\in\bR$, and any L\'evy
process $l$ (and in particular an $\alpha$-stable L\'evy process) there
exists a path-wise unique strong solution of the linear stochastic
differential equation (\ref{eqv}) given by
%
\begin{equation}
\label{eqvsol} v^\varepsilon_t=v_0\ex^{-A t}+
\varepsilon\int_0^t \ex^{-A(t-s)}
\,\di l_s,
\end{equation}
where the last integral is a L\'evy--Wiener integral (see Chapter 4.3.5
in Applebaum \cite{Applebaum09}).
It is helpful to recall another representation of $v^\varepsilon$
which is
obtained with the help of integration by parts, namely
%
\begin{equation}
\label{eqintparts} v^\varepsilon_t=v_0
\ex^{-A t}+\varepsilon l_t-\varepsilon A\int
_0^t \ex^{-A(t-s)}l_s
\,\mathrm{d}s.
\end{equation}
It is clear, that the process $v^\varepsilon$ is also c\`adl\`ag and
its jumps
coincide with the jumps of the driving process $\varepsilon l$.

The equation (\ref{eqx}) for the displacement process $x^\varepsilon
$ can be
also solved explicitly. Applying the Fubini theorem
we obtain
%
\begin{eqnarray}
\label{eqxsol} x^\varepsilon_t&=&x_0+\int
_0^t v^\varepsilon_s \,\di s
=x_0+\int_0^t \biggl[
v_0\ex^{-A s}+\varepsilon\int_0^s
\ex^{-A(s-u)} \,\di l_u \biggr] \,\di s
\nonumber
\\
&=&x_0+\frac{v_0}{A}\bigl(1-\ex^{-A t}\bigr) +\varepsilon
\int_0^t \biggl[ \int_u^t
\ex^{-A(s-u)} \,\di s \biggr] \,\di l_u
\\
&=&x_0+\frac{v_0}{A}\bigl(1-\ex^{-A t}\bigr) +
\frac{\varepsilon}{A} \int_0^t \bigl(1-
\ex^{-A(t-u)} \bigr) \,\di l_u.
\nonumber
\end{eqnarray}
From now on, we set the initial conditions $x_0=v_0=0$.
At the end of the Section \ref{s3}, we discuss the generalization of
the results to the case of arbitrary initial conditions.

For a real valued stochastic process $y=(y_t)_{t\geq0}$ and $a>0$, let
$\tau_a(y)$ denote the first passage time
\[
\tau_a(y)=\inf\{t\geq0\dvt y_t>a\}.
\]
The main goal of this paper is to study the law of the first passage
times $\tau_a(x^\varepsilon)$ of the displacement process
$x^\varepsilon$
in the limit $\varepsilon\to0$.

The asymptotics of $\tau_a(x^\varepsilon)$ can be determined in an especially
simple way in the case of an integrated
OU process driven by a standard Brownian motion $l=b$, that is a
strictly $2$-stable continuous L\'evy process with the
characteristic exponent $\Psi(u)=-u^2/2$, $u\in\bR$.

Consider the Polish space $C([0,\infty),\bR)$ of real-valued continuous
functions endowed with the topology $U$ of local uniform
convergence
associated with the metric
\[
d_U\bigl(x,x'\bigr):=\int_0^\infty
\ex^{-T}\Bigl(1\wedge\sup_{t\in
[0,T]}\bigl|x_t-x'_t\bigr|
\Bigr) \,\di T,\qquad x,x'\in C\bigl([0,\infty)\bigr).
\]
The following result about the weak convergence (denoted in the sequel
by `$\Rightarrow$') of integrated OU processes to the Brownian motion is
well known and is presented here for the sake of
completeness.
%
%
\begin{prp}
\label{pbr}
Let $l=b=(b_t)_{t\geq0}$ be a standard Brownian motion, and
let $x^{\varepsilon}$ be the integrated OU process satisfying the equations
(\ref{eqv}) and (\ref{eqx}) with zero initial conditions.
Then, for any $A>0$,
\[
\bigl(A x^\varepsilon_{t/\varepsilon^2}
\bigr)_{t\geq
0}\Rightarrow(b_t)_{t\geq0}
\]
in $C([0,\infty),\bR;U)$ as $\varepsilon\to0$.
\end{prp}
\begin{pf}
According to (\ref{eqxsol}),
the process $Ax^\varepsilon$ is determined explicitly as
\[
A x^\varepsilon_t= \varepsilon\int_0^t
\bigl(1- \ex^{- A(t-s)}\bigr) \,\di b_s.
\]
Applying the time change $t\mapsto\frac{t}{\varepsilon^2}$ and using the
self-similarity of the Brownian
motion $b$, $\Law(\varepsilon b_{t/\varepsilon^2}, t\geq0)=\Law
(b_t, t\geq0)$, we
obtain that for any $\varepsilon>0$
the process $(Ax^\varepsilon_{t/\varepsilon^2})_{t\geq0}$
coincides in law with the process $AX^\varepsilon$ given by the
convolution integral
%
\begin{equation}
\label{eqAB} AX^\varepsilon_t=\int_0^t
\bigl(1-\ex^{-{A}(t-s)/{\varepsilon^2}}\bigr) \,\di B_s,
\end{equation}
where $B$ is another standard Brownian motion.
We show that the process
\[
Y_t^\varepsilon=AX^{\varepsilon}_t- B_t=
\int_0^t \ex^{-
{A}(t-s)/{\varepsilon^2}} \,\di B_s
\]
converges to zero in probability as $\varepsilon\to0$ uniformly over
$t\in
[0,T]$ for any $T>0$.
Indeed, $Y^\varepsilon$ is a centred Gaussian process with the variance
\[
\E\bigl|Y_t^\varepsilon\bigr|^2= \int_0^t
\ex^{-{2A}(t-s)/{\varepsilon
^2}} \,\di s= \frac{\varepsilon^2}{2A}\bigl(1-\ex^{-{2A}t/{\varepsilon^2}}\bigr)\leq
\frac{\varepsilon^2}{2A},\qquad t\in[0,T]. 
\]
Applying Theorem 5.3 from Adler \cite{Adler90},
we conclude that for any $T>0$ there is an absolute constant $C>0$ such
that for any $\Delta>0$
\[
\P \Bigl(\sup_{t\in[0,T]} \bigl|Y^\varepsilon_t \bigr|>\Delta \Bigr) \leq C\Delta
\biggl(1-\Phi \biggl(\frac{2A\Delta}{\varepsilon^2 } \biggr) \biggr)\to
0,\qquad \varepsilon\to0,
\]
where $\Phi$ is the probability distribution function of a standard
Gaussian random variable.
The weak convergence of $(Ax^\varepsilon_{t/\varepsilon^2})_{t\geq
0}\Rightarrow b$ in
$C([0,\infty),\bR;U)$ follows
immediately from the convergence in probability.
\end{pf}

It is clear that the law of the first passage time $\tau_a(x^\varepsilon)$ is
determined with the help of the running supremum
of the process $x^\varepsilon$,
\[
S_t^\varepsilon:=\sup_{s\in[0,t]} x^\varepsilon_s,\qquad
t\geq0,
\]
namely $\P(\tau_a(x^\varepsilon)\leq t)=\P(S_t^\varepsilon\geq a)$.
Since the running supremum of a continuous process is a continuous
mapping in $C([0,\infty),\bR;U)$, we obtain the convergence in law of the
first passage times.
%
\begin{cor}
\label{cbr}
For any $a>0$
\[
\varepsilon^2 \tau_a\bigl(x^\varepsilon\bigr)
\stackrel{d} {\to} \tau_{
{a}/{A}}(b)\qquad\mbox{as }\varepsilon \to0.
\]
\end{cor}
The probability density of the first passage time $\tau_{
{a}/{A}}(b)$ is well known,
\[
\P\bigl(\tau_{{a}/{A}}(b)\leq t\bigr)=\frac{a}{A\sqrt{2\uppi }}\int
_0^t \frac
{1}{s^{3/2}}\ex^{-{a^2}/({2A^2 s})} \,\di s.
\]

If the driving L\'evy process $l=l^{(\alpha)}$ is $\alpha$-stable and
non-Gaussian,
the situation becomes more complicated.
Consider the space $D([0,\infty),\bR)$ of real valued c\`adl\`ag functions.
We shall see in Proposition \ref{pfdd} that $Ax^\varepsilon_{\cdot
/\varepsilon^\alpha}$
converges to $l^{(\alpha)}$
in the sense of finite-dimensional distributions whereas
the integrated OU process $Ax^\varepsilon$ is absolutely continuous.
Thus no weak convergence in the space $D([0,\infty),\bR)$ equipped with
the topology of the local uniform convergence is possible.

In his seminal paper, Skorohod \cite{Skorokhod56} introduced four
weaker topologies on the space $D([0,\infty),\bR)$ different
from the uniform topology.
The most frequently used topology $J_1$ is designed to match the jump
times and sizes
of the approximating processes and their limit,
and does not fit in with our setting. Thus, we shall prove convergence
in the weaker $M_1$-topology
which is still strong enough
to guarantee the continuity of the supremum, and thus the convergence
of the first passage times.
Essentially this topology linearises the jumps through the introduction
of a fictitious
time-scale and is appropriate for establishing the convergence of a
sequence of continuous processes
to a discontinuous limit.
It is also worth mentioning that the idea of a fictitious time-scale
has been used in
some other contexts, see Williams \cite{Williams01} and the references therein.
A very detailed treatment of the $M_1$-convergence
can be found in the monograph by
Whitt \cite{whitt02}.

Let us recall the necessary facts about the space $D([0,\infty),\bR)$
endowed with the non-uniform topology $M_1$.
For\vspace*{1pt} any function
$x\in D([0,\infty),\bR)$ and for any $T>0$ define a \textit{completed
graph} $\Gamma_x^T$ of the restriction of $x$ on $[0,T]$ as a set
\[
\Gamma_x^T:=\bigl\{(x_0,0)\bigr\}\cup\bigl
\{(z,t)\in\bR\times(0,T]\dvt z=c x_{t-}+(1-c)x_t\mbox{ for
some } c, c\in[0,1]\bigr\}.
\]
The completed graph is a subset in $\bR^2$ containing the graph of $x$
on $[0,T]$ as well as
the line segments connecting the points of discontinuity $(x_{t-},t)$
and $(x_t,t)$.
On a completed graph we introduce an order saying that
$(z,t)\leq(z',t')$ if either $t< t'$ or $t=t'$ and \mbox{$|x_{t-}-z|\leq
|x_{t-}-z'|$}.
A parametric representation of the graph is a continuous
mapping\vspace*{1pt}
$(z_u,t_u)\dvt[0,1]\to\Gamma_x^T$,
which in non-decreasing w.r.t. order on the completed graph. Denote
$\Pi_x^T$ the set of
all parametric representations of the graph $\Gamma_x^T$. The Skorokhod
$M_1$-topology in $D([0,\infty),\bR)$ is then induced
by the metric
\begin{eqnarray*}
d_{M_1}\bigl(x,x'\bigr)&:=&\int_0^\infty
\ex^{-T} \bigl(1\wedge d_{M_1,T}\bigl(x,x'\bigr)
\bigr) \,\di T,
\\
d_{M_1,T}\bigl(x,x'\bigr)&:=&\mathop{\inf_{(z,t)\in\Pi_{x}^T,}}_{(z',t')\in
\Pi
_{x'}^T}
\max_{u\in[0,1]}\bigl\{ \bigl|z_u-z'_u\bigr|,
\bigl|t_u-t'_u\bigr| \bigr\},\qquad x,x'\in
D\bigl([0,\infty),\bR\bigr), T>0,
\end{eqnarray*}
see Whitt \cite{whitt02}, Sections 3.3, 12.3 and 12.9. One can
construct a metric equivalent to $d_{M_1}$, for which the
space $D([0,\infty),\bR; M_1)$ is Polish, see Whitt \cite{whitt02},
Section 12.8.

The main result of this paper is the following convergence result.
%
\begin{theorem}
\label{tm}
Let $l^{(\alpha)}=(l^{(\alpha)}_t)_{t\geq0}$ be an $\alpha$-stable
L\'evy process, $\alpha\in(0,2)$ and
let $x^{\varepsilon}$ be the integrated OU process satisfying the equations
(\ref{eqv}) and (\ref{eqx}) with zero initial conditions.
Then
\[
\bigl(A x^\varepsilon_{{t}/{\varepsilon^\alpha}} \bigr)_{t\geq0} \Rightarrow
\bigl(l^{(\alpha)}_t\bigr)_{t\geq0}
\]
in $D([0,\infty),\bR; M_1)$ as $\varepsilon\to0$.
\end{theorem}

The convergence of the first passage times follows immediately.
%
\begin{cor}
Let $l^{(\alpha)}$ be an $\alpha$-stable process with $\limsup_{t\to
\infty} l^{(\alpha)}_t=+\infty$ a.s.
Then, for any $a>0$
\[
\varepsilon^\alpha\tau_a\bigl(x^\varepsilon\bigr)
\stackrel{d} {\to}\tau_{{a}/{A}}\bigl(l^{(\alpha)}\bigr) \qquad\mbox{as }
\varepsilon\to0.
\]
\end{cor}
\begin{pf}
As in Corollary \ref{cbr}, we define the first exit time with the help
of the running supremum
$S_t^\varepsilon:=\sup_{s\in[0,t]} x^\varepsilon_t$, $t\geq0$,
and the inverse function
$I_t^\varepsilon:=\inf\{s\geq0\dvt S^\varepsilon_s>t\}$, $t\geq0$.
Under the condition $x_0=v_0=0$, the inverse function is continuous in
the $M_1$-topology,
see Puhalskii and Whitt \cite{PuhWhi97}, Lemma 2.1. The continuous
mapping theorem yields the result.
\end{pf}

In contrast to the Brownian case, the laws of the first passage times
$\tau_{a}(l^{(\alpha)})$ of an
$\alpha$-stable L\'evy process are often not known explicitly.
We refer the reader to the recent works by Kuznetsov \cite
{Kuznetsov11} and
Simon \cite{Simon11},
and references therein for various results on this topic.

The rest of the paper is organized as follows. Since the $\alpha
$-stable case can be studied with the help of an appropriate
time change, which transforms the small noise amplitude into the big
friction parameter,
we shall study the $M_1$-convergence of continuous integrated OU
processes with big friction parameter driven by arbitrary
L\'evy processes. This result
can be of its own interest. Finally, we prove Theorem \ref{tm} and
discuss the case of arbitrary initial conditions.

\section{Absolutely continuous aproximations of L\'evy processes in
$M_1$-topology and the proof of the main result}\label{s3}

For $A>0$, $\gamma\geq0$,
and a real valued L\'evy process $L$ with a characteristic exponent
$\Psi$ given by (\ref{eqPhi})
we study the system of stochastic differential equations
%
\begin{eqnarray}
\label{eqVg} V_t^\gamma&=&-\gamma A\int_0^t
V_s^\gamma\,\di s+L_t,
\\
\label{eqXg} X^\gamma_t&=&\gamma\int_0^t
V_s^\gamma\,\di s.
\end{eqnarray}
First, we prove the convergence of finite-dimensional
marginals of $AX^\gamma$ to those of $L$ in probability.
%
%
\begin{prp}
\label{pfdd}
For any $m\geq1$ and $0\leq t_1<\cdots<t_m<\infty$,
\[
\bigl(A X_{t_1}^\gamma,\ldots, A X^\gamma_{t_m}
\bigr) \stackrel{\P} {\to} (L_{t_1},\ldots, L_{t_m} )\qquad\mbox{as
}\gamma\to \infty.
\]
\end{prp}
\begin{pf}
With the help of the formulae (\ref{eqvsol}) and (\ref{eqxsol}) one
can solve the equations
(\ref{eqVg}) and (\ref{eqXg}) explicitly:
%
\begin{equation}
\label{eqAX} V_t^\gamma= 
\int
_0^t \ex^{-\gamma A (t-s)} \,\di L_s,\qquad
AX^\g_t = \int_0^t
\bigl(1- \ex^{-\gamma A(t-s)} \bigr) \,\di L_s.
\end{equation}
It is clear that the processes $AX^\gamma$ and $V^\gamma$ start at the
origin a.s.,
$A X^\gamma_0=V^\gamma_0=L_0=0$.
For $m\geq1$ fix the time instants $0=t_0<t_1<\cdots<t_m<\infty$
and real numbers $u_0, u_1,\ldots,u_m$ and
consider the characteristic function
%
\begin{equation}
\label{eqchF} \E\exp \Biggl(\i\sum_{k=0}^m
u_k \bigl(A X^\gamma_{t_k}-L_{t_{k}}
\bigr) \Biggr) =\E\exp \Biggl(- \i\sum_{k=1}^m
u_k\int_0^{t_k} \ex^{-\gamma
A(t_k-s)}
\,\di L_s \Biggr).
\end{equation}
We represent the sum in the last exponent as a sum of independent
random variables
\begin{eqnarray*}
\sum_{k=1}^m u_k \bigl(A
X^\gamma_{t_k}-L_{t_k}\bigr)&=&-\sum
_{k=1}^m u_k \int_{0}^{t_1}
\ex^{-\gamma A(t_k-s)} \,\di L_s
\\
&&{}- \sum_{k=2}^m u_k \int
_{t_1}^{t_2} \ex^{-\gamma A(t_k-s)} \,\di L_s-
\cdots -u_m \int_{t_{m-1}}^{t_m}
\ex^{-\gamma A(t_m-s)} \,\di L_s
\end{eqnarray*}
and show that the characteristic function of every summand converges to
$1$ as $\gamma\to\infty$.
Fix an index $j$, $1\leq j\leq m$. Then by a well-known formula for
characteristic functions of
convolution integrals w.r.t. a L\'evy process (see, e.g.,
Lemma 17.1 in Sato \cite{Sato-99}) we obtain
the equality
%
\begin{eqnarray}
\label{eqlt}
&&\ln\E\exp \Biggl(-\i\sum_{k=j}^m
u_k \int_{t_{j-1}}^{t_j} \ex^{-\gamma
A(t_k-s)}
\,\di L_s \Biggr)\nonumber\\[-8pt]\\[-8pt]
&&\quad = \int_{t_{j-1}}^{t_j} \Psi
\Biggl( -\sum_{k=j}^m u_k
\ex^{-\gamma
A(t_k-s)} \Biggr) \,\di s.\nonumber
\end{eqnarray}
For brevity, we denote the argument
\[
u^\gamma_j(s):=-\sum_{k=j}^m
u_k \ex^{-\gamma A(t_k-s)},\qquad t_{j-1}\leq s\leq
t_j,
\]
and $u^*=\sum_{k=1}^m |u_k|<\infty$. Clearly, $|u^\gamma_j(s)|\leq
\sum_{k=j}^m |u_k|\leq u^*$ for $s\in[t_{j-1}, t_j]$, $1\leq j\leq m$.
The exponent $\Psi(u)$ is continuous and bounded on $[-u^*,u^*]$.
On each of the intervals $[t_{j-1},t_j)$, $1\leq j\leq m$, we determine
the pointwise limit of $\Psi(u^\gamma_j(s))$ as
$\gamma\to\infty$, namely
\[
\lim_{\gamma\to\infty} \Psi\bigl(u^\gamma_j(s)\bigr)=
\Psi(0)=0.
\]
After applying the Lebesgue dominated convergence theorem to the
right-hand side of (\ref{eqlt}),
we conclude that the term on the left-hand side of (\ref{eqchF}) tends
to $1$ as $\gamma\to\infty$, and as a well-known
consequence the convergence in probability
\[
\bigl(A X_0^\gamma,A X_{t_1}^\gamma,\ldots,
A X^\gamma_{t_m} \bigr) \stackrel{\P} {\to}
(L_0,L_{t_1},\ldots, L_{t_m} )
\]
holds as $\gamma\to\infty$.
\end{pf}

In the proof of the next Theorem \ref{thapp} about the convergence of
$A X^\gamma$ to $L$ in the $M_1$-topology, we
shall make use of the following oscillation function.
For $x,y\in\bR$ denote the segment $\llbracket x,y\rrbracket:=\{z\in
\bR
\dvt z=x+c(y-x), c\in[0,1]\}$
and introduce the oscillation
function $M\dvtx\bR^3\to[0,\infty)$,
\[
M(x_1,x,x_2):=\cases{\displaystyle  \min\bigl\{|x-x_1|,|x_2-x|
\bigr\},&\quad if $x\notin\llbracket x_1,x_2\rrbracket $,
\vspace*{2pt}\cr
0, &
\quad$x\in\llbracket x_1,x_2\rrbracket$. }
\]
In other words, $M(x_1,x,x_2)$ is the Euclidean distance between the
point $x$ and the segment $\llbracket x_1,x_2\rrbracket$.

Now we prove the main result of this section.
%
\begin{theorem}
\label{thapp}
Let $L$ be an arbitrary real valued L\'evy process, and $X^\gamma$ be a
solution of (\ref{eqXg}). Then for any $A>0$,
\[
A X^\gamma\stackrel{\P} {\to}L\qquad\mbox{in } D\bigl([0,\infty),\bR;
M_1\bigr)\mbox{ as }\gamma\to\infty.
\]
\end{theorem}
\begin{pf}
1. First, with the help of the L\'evy--It\^o decomposition we represent
$L$ as a sum of a
continuous Brownian motion $\sigma B$ and a L\'evy process $Z$ without
Gaussian part.
Due to the linearity of the Langevin equation, we represent the
solution $V^\gamma$ as a sum
\[
V^\gamma_t=\sigma\int_0^t
\ex^{-\gamma A(t-s)} \,\di B_s +\int_0^t
\ex^{-\gamma A(t-s)} \,\di Z_s,
\]
and consequently the process $AX^\gamma$ as a sum of two continuous processes
%
\begin{eqnarray}
\label{eqsum} A X^\gamma_t&=& AX^{\gamma, B}_t+
AX^{\gamma, Z}_t \nonumber\\[-8pt]\\[-8pt]
&:=& \sigma\int_0^t
\bigl(1- \ex^{-\gamma A(t-s)}\bigr) \,\di B_s
+\int
_0^t \bigl(1-\ex^{-\gamma A(t-s)}\bigr) \,\di
Z_s.\nonumber
\end{eqnarray}
In Proposition \ref{pbr}, we proved that $AX^{\gamma,B}$ converges to
$\sigma B$ in probability in the local uniform topology
(see (\ref{eqAB}) with $\gamma=\frac{1}{\varepsilon^2}$),
and consequently in the $M_1$-topology.
Since $\sigma B$ is continuous, due to Corollary 12.7.1
in Whitt \cite{whitt02} it is sufficient to prove the $M_1$-convergence
of $AX^{\gamma, Z}$ to $Z$.
The convergence in probability of finite-dimensional marginals of
$AX^{\gamma, Z}$ follows from Proposition \ref{pfdd}.
The L\'evy process $Z$ is stochastically continuous at any $T\geq0$,
so that due to
Section 3 in the original paper by Skorohod \cite{Skorokhod56} or
Chapter 12 in Whitt \cite{whitt02} for the convergence in $D([0,\infty
),\bR;M_1)$
it is sufficient to establish
the boundedness of the family $\{AX^{\gamma, Z}\}$, that is to show
that for every $T>0$
%
\begin{equation}
\label{eqbound} \lim_{K\to\infty}\sup_{\gamma>0}\P \Bigl(
\sup_{t\in
[0,T]}\bigl|AX^{\gamma,Z}_t\bigr|>K \Bigr)=0;
\end{equation}
and
to control the oscillation function, that is to show that for every
$T>0$ and $\Delta>0$
%
\begin{equation}
\label{eqM} \lim_{\delta\downarrow0}\limsup_{\gamma\to\infty}\P \Bigl(\mathop{
\sup_{0\leq t_1<t<t_2\leq T,}}_{t_2-t_1\leq\delta} M\bigl(AX^{\gamma,Z}_{t_1},AX^{\gamma,Z}_t,AX^{\gamma,Z}_{t_2}
\bigr)>\Delta \Bigr)=0.
\end{equation}
Without loss of generality, we assume from now on that $A=1$. Let $T>0$
be fixed.

2. For the proof of (\ref{eqbound}), we use the representation (\ref
{eqAX}) of $X^{\gamma,Z}$. Integrating
by parts (compare with (\ref{eqintparts})) yields
\[
X^{\gamma,Z}_t=Z_t-\int_{0}^t
\ex^{-\gamma(t-s)} \,\di Z_s=\gamma \int_0^t
\ex^{-\gamma(t-s)}Z_s \,\di s.
\]
Thus for any $\gamma\geq0$, we obtain the estimate
%
\begin{equation}
\label{eqXZ} \sup_{t\in[0,T]}\bigl|X^{\gamma,Z}_t\bigr|\leq
\sup_{t\in[0,T]}|Z_t|\sup_{t\in
[0,T]}\gamma\int
_{0}^t \ex^{-\gamma(t-s)} \,\di s \leq
\sup_{t\in[0,T]} |Z_t|,
\end{equation}
so that the condition (\ref{eqbound}) holds true.

3. We now prove the estimate (\ref{eqM}). Let $\Delta>0$ be fixed. We
show that
for any $\theta>0$ there is $\delta_0=\delta_0(\Delta, \theta, T)$ such
that for any
$\delta\in(0,\delta_0]$ there is $\gamma_0=\gamma_0(\delta,
\Delta,
\theta, T)$ such that for all
$\gamma>\gamma_0$ the inequality
\[
\P \Bigl(\mathop{\sup_{0\leq t_1<t<t_2\leq T,}}_{t_2-t_1\leq\delta} M\bigl(
X^{\gamma,Z}_{t_1}, X^{\gamma,Z}_t,
X^{\gamma,Z}_{t_2} \bigr)>\Delta \Bigr)\leq\theta
\]
holds true. The proof of this inequality will consist of three steps.

\textit{Step} 1. \textit{Reduction to a compound Poisson process with drift.}
First, we decompose $Z$ into a sum of a martingale with bounded jumps
and small variance and a compound
Poisson process with drift.

Let $a=a(\Delta, \theta,T)\in(0,1]$ be such that $\nu(\{a\})= \nu
(\{-a\}
)=0$ and
\[
\frac{16T}{\Delta^2}\int_{\{|y|<a\}}y^2 \nu(\di y)\leq
\frac
{\theta}{4}.
\]
For this $a$, denote
\[
\mu_a:=\mu-\int_{a\leq|y|\leq1} y \nu(\di y)
\]
and consider the processes
\[
\eta_t:=\sum_{s\leq t}\Delta
Z_s\bI\bigl(|\Delta Z_s|\geq a\bigr)+\mu_a t
\quad\mbox{and}\quad \xi_t:=Z_t-\eta_t.
\]
The processes $\xi$ and $\eta$ are independent L\'evy processes with
the respective L\'evy--Khintchine representations
\begin{eqnarray*}
\E\ex^{\i u \xi_1} &=&\exp \biggl(\int_{\{|y|< a\}}\bigl(
\ex^{\i uy}-1-\i uy\bigr) \nu(\di y) \biggr),
\\
\E\ex^{\i u \eta_1} &=&\exp \biggl(\int_{\{|y|\geq a\}}\bigl(
\ex^{\i uy}-1\bigr) \nu(\di y) +\i\mu_a u \biggr),\qquad u\in\bR.
\end{eqnarray*}
Moreover, $Z=\xi+\eta$, $\eta$ is a compound Poisson process with the
drift $\mu_a$, and $\xi$ is a zero mean martingale with
the variance
$\E\xi^2_t=t\int_{\{|y|<a\}}y^2 \nu(\di y)$.
Due to the linearity of equations (\ref{eqVg}) and (\ref{eqXg}), we
obtain the representation
\[
X^{\g,Z}=X^{\gamma, \xi}+X^{\gamma, \eta}
\]
with
\[
X^{\gamma, \xi}_t=\int_0^t
\bigl(1- \ex^{-\gamma(t-s)}\bigr) \,\di\xi_s,\qquad X^{\gamma, \eta}_t=
\int_0^t\bigl(1- \ex^{-\gamma(t-s)}\bigr) \,\di
\eta_s.
\]
Denote the event
\[
E_\xi:= \biggl\{\omega\dvt\sup_{t\in[0,T]}\bigl|X^{\gamma,\xi}_t\bigr|<
\frac
{\Delta}{4} \biggr\}.
\]
Using the estimate similar to (\ref{eqXZ}) and applying the Doob
inequality to the martingale $\xi$ we obtain for all $\gamma\geq0$
that
\begin{eqnarray*}
\P\bigl( E_\xi^c \bigr)&\leq&\P \biggl(
\sup_{t\in[0,T]}|\xi_t|>\frac{\Delta
}{4} \biggr)\leq
\frac{16T\E|\xi_T|^2 }{\Delta^2}\\
&\leq&\frac{16T}{\Delta^2} \int_{\{|y|<a\}}y^2
\nu(\di y) \leq\frac{\theta}{4}.
\end{eqnarray*}
Thus for all $\omega\in E_\xi$ and for all $\gamma\geq0$ the inequality
\[
\sup_{t\in[0,T]}\bigl|X^{\gamma,Z}_t-X^{\gamma, \eta}_t
\bigr|\leq\sup_{t\in
[0,T]}\bigl|X^{\gamma, \xi}_t \bigr|< \frac{\Delta}{4}
\]
holds true. This implies that for all $\gamma\geq0$ and $0\leq
t_1<t<t_2\leq T$
\[
\bigl|M\bigl(X^{\gamma,Z}_{t_1},X^{\gamma,Z}_t,X^{\gamma,Z}_{t_2}
\bigr) -M\bigl(X^{\gamma, \eta}_{t_1},X^{\gamma, \eta}_t,X^{\gamma, \eta
}_{t_2}
\bigr) \bigr|\leq\frac{\Delta}{2}.
\]

\textit{Step} 2. \textit{Local extrema of $X^{\gamma,\eta}$.}
There exists a level $z=z(a,\theta,T)>0$ such that for the event
\[
E_\eta:= \biggl\{ \sup_{t\in[0,T]}|\eta_t|\leq
\frac{z}{2} \biggr\}
\]
the inequality
\[
\P\bigl(E_\eta^c\bigr)\leq\frac{\theta}{4}
\]
holds.
In particular, this implies that for $\omega\in E_\eta$ the jump sizes
of $\eta$ do not exceed $z$ in absolute value, that is
$\sup_{t\in[0,T]} |\Delta\eta_t(\omega)|\leq z$.
The process $\eta$ has the finite jump intensity
\[
\beta_a=\int_{|y|\geq a}\nu(\di y)<\infty.
\]
For the L\'evy process $\eta$, consider its counting jump process
$N=(N_t)_{t\geq0}$ which is
a Poisson process with
intensity $\beta_a$. Denote by $\{\tau_k\}_{k\geq0}$ the sequence of
arrival times of $\eta$, $\tau_0=0$, and by
$\{J_k\}_{k\geq0}$, $J_0=0$, the sequence of its jump sizes,
$J_k:=\eta_{\tau_k}-\eta_{\tau_k-}$.
It is easy to see that the process $X^{\gamma,\eta}$ has the following
path-wise representation:
%
\begin{equation}
\label{eqxeta} X^{\gamma, \eta}_t=\sum_{k=0}^{N_t }J_k
\bigl(1-\ex^{-\gamma(t-\tau_k)} \bigr) 
+ \mu_a \biggl(t-
\frac{1-\ex^{-\gamma t}}{\gamma} \biggr),\qquad t\geq0.
\end{equation}
We choose $m^*\geq0$ such that
\[
\P\bigl(N_T\leq m^*\bigr)>1-\frac{\theta}{4}.
\]
Further, for $\kappa>0$ and $m=0,\ldots, m^*$ consider the events
\begin{eqnarray*}
C^0_\kappa&:=&C^0=\{\omega\dvt
N_T=0\},
\\
C^m_\kappa&:=& \{\omega\dvt N_T=m\}\\
&&{} \cap \{
\omega\dvt\tau_{k}-\tau_{k-1}\geq2\kappa\mbox{ for } k=1,
\ldots,m\mbox{, and }T-\tau_m\geq2\kappa\},\qquad
1\leq m\leq m^*,
\\
C_\kappa&:=&\bigsqcup_{m=0}^{m^*}C_m
\subset\bigl\{N_T\leq m^*\bigr\}.
\end{eqnarray*}
It is well known (see, e.g., Proposition 3.4 in Sato \cite
{Sato-99}) that conditioned on $\{N_T=m\}$, the jump times
$\tau_1,\ldots,\tau_m$ are distributed on the interval $[0,T]$ with the
probability law of the order statistics obtained from $m$ samples of the
uniform distribution on $[0,T]$.
Thus, we are able to choose $\kappa=\kappa(\theta,\Delta, T,m^*)>0$
small enough, such that
\[
\P(C_\kappa)> 1-\frac{\theta}{3}.
\]
For a fixed $m=0,\ldots,m^*$ consider $\omega\in E_\eta\cap
C^m_\kappa$.
It is easy to see from the representation (\ref{eqxeta}) that
\[
X^{\gamma, \eta}_t(\omega)= \cases{\displaystyle  \mu_a \biggl(t-
\frac{1-\ex^{-\gamma t}}{\gamma} \biggr), \qquad t\in \bigl[0,\tau_1(\omega)\bigr),
\vspace*{2pt}\cr
\displaystyle \sum
_{j=1}^{k}J_j(\omega) \bigl(1-
\ex^{-\gamma(t-\tau
_j(\omega))} \bigr) +\mu_a \biggl(t-\frac{1-\ex^{-\gamma t}}{\gamma} \biggr),
\vspace*{2pt}\cr
\hspace*{103pt}\displaystyle t\in\bigl[\tau_{k-1}(\omega ),\tau_k(\omega)\bigr),  k=2,\ldots,
m,
\vspace*{2pt}\cr
\displaystyle \sum_{j=1}^{m}J_j(
\omega) \bigl(1-\ex^{-\gamma(t-\tau
_j(\omega))} \bigr) + \mu_a \biggl(t-
\frac{1-\ex^{-\gamma t}}{\gamma} \biggr), \vspace*{2pt}\cr
\qquad\hspace*{83pt}t\in \bigl[\tau_{m}(\omega),T\bigr]. }
\]
The process $X^{\gamma,\eta}$ has smooth paths on the intervals
$[\tau_{k-1},\tau_{k})$, $k=1,\ldots,m$, and $[\tau_m, T]$.
We show that for $\gamma$ large enough the paths of $X^{\gamma,\eta}$
are either monotone on these intervals, or have at
most one local extremum on each of the intervals.
Indeed, $X^{\gamma,\eta}$ is obviously monotone on
$t\in[0,\tau_1]$.

Let now $1\leq m\leq m^*$. For
$t\in(\tau_{k},\tau_{k+1})$, $k=1,\ldots,m-1$, and for $t\in(\tau_{m},T)$ consider the
derivative of $X^{\gamma,\eta}$ w.r.t. $t$:
\begin{eqnarray*}
\frac{\di}{\di t}X^{\gamma,\eta}_t
&=&\gamma
\sum_{j=1}^k J_j
\ex^{-\gamma(t-\tau_j)} +\mu_a\bigl(1-\ex^{-\gamma t}\bigr)
\\
&=&\gamma\sum_{j=1}^{k-1} J_j
\ex^{-\gamma(t-\tau_j)} +\gamma J_k\ex^{-\gamma(t-\tau_k)} +\mu_a
\bigl(1-\ex^{-\gamma t}\bigr)
\\
&=& \gamma J_k\ex^{-\gamma(t-\tau_k)} \Biggl(1+ \sum
_{j=1}^{k-1} \frac
{J_j}{J_k}\ex^{-\gamma(\tau_k-\tau_j)}
\Biggr) +\mu_a\bigl(1-\ex^{-\gamma t}\bigr).
\end{eqnarray*}
Taking into account that the jump sizes $J_k$ are bounded,
$|J_k(\omega)|\in[a, z]$ and the arrival times are separated by
$2\kappa$,
$\tau_k-\tau_j\geq2(k-j)\kappa$, $1\leq j\leq k-1$, and $T-\tau_m\geq
2\kappa$, we can
choose a non-random $\gamma_m=\gamma_m(a,z,\kappa, m)$ such that for
$\gamma\geq\gamma_m$
the equation
\begin{eqnarray*}
\frac{\di}{\di t}X^{\gamma,\eta}_t&=& \gamma J_k
\ex^{-\gamma(t-\tau_k)} \Biggl(1+ \sum_{j=1}^{k-1}
\frac
{J_j}{J_k}\ex^{-\gamma(\tau_k-\tau_j)} \Biggr) +\mu_a\bigl(1-
\ex^{-\gamma t}\bigr)=0
\end{eqnarray*}
has at most one solution on each of the intervals $(\tau_{k},\tau_{k+1})$, $k=1,\ldots,m-1$, and on $(\tau_m,T)$.
This unique solution (the local extremum of $X^{\gamma,\eta}$) exists
if and only if $\mu_a\neq0$ and
$\frac{J_k}{\mu_a}<0$, and is located at
%
\begin{eqnarray}
\label{eqt*} t^*_k&=&t^*_k(\gamma)=\frac{1}{\gamma}
\ln \Biggl(1+\gamma\ex^{\gamma
\tau
_k} \biggl|\frac{J_k}{\mu_a} \biggr| \Biggl(1+\sum
_{j=1}^{k-1} J_j\ex^{-\gamma(\tau_k-\tau_j)}
\Biggr) \Biggr)
\nonumber\\[-8pt]\\[-8pt]
&\approx&\tau_{k}+\frac{1}{\gamma}\ln \biggl( \gamma \biggl|
\frac
{J_k}{\mu_a} \biggr| \biggr),\qquad 1\leq k\leq m.
\nonumber
\end{eqnarray}
Moreover, we can choose $\gamma_m$ big enough such that for $\gamma
\geq
\gamma_m$
we have $\tau_k<t^*_k\leq\tau_k+\kappa$ for all
$k=1,\ldots,m$.
Furthermore, we choose $\gamma_m$ big enough such that for $\gamma
\geq
\gamma_m$
%
\begin{eqnarray}
\label{eqw1} \max_{t\in[\tau_{k}+\kappa,\tau_{k+1}]} \biggl|\frac{\di}{\di
t}X^{\gamma,\eta}_t-
\mu_a \biggr|&\leq&\frac{\Delta}{4},\qquad k=1,\ldots, m-1, \quad\mbox{and}
\nonumber\\[-8pt]\\[-8pt]
\max_{t\in[\tau_m+\kappa,T]} \biggl|\frac{\di}{\di t}X^{\gamma,\eta
}_t-
\mu_a \biggr|&\leq& \frac{\Delta}{4}.
\nonumber
\end{eqnarray}
Additionally for $\mu_a\neq0$, we can assume that for $\gamma\geq
\gamma_m$
%
\begin{eqnarray}
\label{eqw2}
\max_{t\in[t^*_k,\tau_{k+1}]} \biggl|\frac{\di}{\di t}X^{\gamma,\eta
}_t
\biggr|&\leq&2|\mu_a| \qquad\mbox{if } \frac{J_k}{\mu_a}<0, k=1,\ldots, m-1,
\quad\mbox{and }
\nonumber\\[-8pt]\\[-8pt]
\max_{t\in[t^*_m,T]} \biggl|\frac{\di}{\di t}X^{\gamma,\eta}_t \biggr|&\leq& 2|
\mu|_a \qquad\mbox{if } \frac{J_m}{\mu_a}<0.
\nonumber
\end{eqnarray}
Overall, for $\gamma\geq\gamma_m$ and for $\omega\in E_\eta\cap
C^m_\kappa$
the paths of $X^{\gamma,\eta}$ have the following structure: they are
continuous on $[0,T]$, smooth on
$(\tau_k,\tau_{k+1})$, $k=0,\ldots, m$, and $(\tau_m,T)$ and may have
extrema either at arrival times
$\tau_k$, $k=1,\ldots, m$, or at time instants $t^*_k$ given by (\ref
{eqt*}) provided $\frac{J_k}{\mu_a}<0$. The slope of
$X^{\gamma,\eta}$ is close to $\mu_a$ on the left-hand neighbourhoods
of the arrival times
$\tau_k$, $k=1,\ldots, m$, and $T$.
The derivative of $X^{\gamma,\eta}$ is bounded by a constant, say
$2|\mu_a|$, in the right-hand neighbourhoods
of the local extrema $t^*_k$.
Let $\gamma^*:=\bigvee_{m=1}^{m^*}\gamma_m$.

\textit{Step} 3. \textit{Estimate of the oscillation function $M$.}
Let $\delta_0\in(0,\kappa\wedge\frac{\Delta}{8(|\mu_a|+1)})$,
$\gamma\geq\gamma^*$, and let
$\omega\in E_\eta\cap C^m_\kappa$ for some $m=0,\ldots, m^*$.

We estimate the value of the oscillation function
$M=M(X^{\gamma,\eta}_{t_1},X^{\gamma,\eta}_t,X^{\gamma,\eta}_{t_2})$
for $0\leq t_1<t<t_2\leq T$ and
$t_2-t_1\leq\delta\leq\delta_0$. Let us consider three cases:

\begin{longlist}[(iii)]
\item[(i)] If the path of $t\mapsto X^{\gamma,\eta}_t(\omega)$ is
monotone on $[t_1,t_2]$, then
$M =0$.

\item[(ii)]
Let $\tau_k\in[t_1,t_2]$ for some $k=1,\ldots,m$, and let $\tau_k$
be a
local extremum. In this case, the maximal value of
$M$ over $t\in[t_1,t_2]$ is attained at $\tau_k$ and
\begin{eqnarray*}
M&\leq&\min\bigl\{ \bigl|X_{\tau_k}^{\gamma,\eta}- X_{t_1}^{\gamma,\eta}\bigr|,
\bigl|X_{\tau
_k}^{\gamma,\eta}- X_{t_2}^{\gamma,\eta}\bigr| \bigr\}
\leq\bigl|X_{\tau_k}^{\gamma,\eta}- X_{t_1}^{\gamma,\eta}\bigr|\\
&\leq&\bigl|X_{\tau_k}^{\gamma,\eta}- X_{\tau_k-\delta_0}^{\gamma,\eta}\bigr|.
\end{eqnarray*}
Then due to (\ref{eqw1})
\[
M\leq \biggl(|\mu_a|+\frac{\Delta}{4} \biggr)\delta_0
\leq\frac
{\Delta}{4}.
\]

\item[(iii)] Let $t\mapsto X^{\gamma,\eta}$ be non-monotone in
$[t_1,t_2]$ and a local extremum $t^*_k$ exist
and belong to $[t_1,t_2]$ for some $k=1,\ldots,m$.
Then we estimate with the help of (\ref{eqw2}) that
\begin{eqnarray*}
M&\leq&\min\bigl\{ \bigl|X_{t^*_k}^{\gamma,\eta}- X_{t_1}^{\gamma,\eta}\bigr|,
\bigl|X_{t^*_k}^{\gamma,\eta}- X_{t_2}^{\gamma,\eta}\bigr| \bigr\}
\leq\bigl|X_{t^*_k}^{\gamma,\eta}-
X_{t^*_k+\delta_0}^{\gamma,\eta}\bigr|\\
&\leq&
2 |\mu_a|\delta_0\leq\frac{\Delta}{4}.
\end{eqnarray*}
\end{longlist}
Overall, these estimates imply, that for all $0<\delta\leq\delta_0$
and $\gamma\geq\gamma^*$
\[
\P \biggl(\mathop{\sup_{0\leq t_1<t<t_2\leq T,}}_{t_2-t_1\leq\delta} M\bigl(X^{\gamma,\eta}_{t_1},X^{\gamma,\eta}_{t},X^{\gamma,\eta}_{t_2}
\bigr)> \frac{\Delta}{2} \Big| E_\eta\cap C_\kappa \biggr)=0
\]
and the inequality (\ref{eqM}) follows:
\begin{eqnarray*}
&&
\P \Bigl(\mathop{\sup_{0\leq t_1<t<t_2\leq T,}}_{t_2-t_1\leq\delta} M\bigl(X^{\gamma,Z}_{t_1},X^{\gamma,Z}_{t},X^{\gamma,Z}_{t_2}
\bigr) >\Delta \Bigr)
\\
&&\quad
\leq\P \biggl( \mathop{
\sup_{0\leq t_1<t<t_2\leq T,}}_{t_2-t_1\leq
\delta} M\bigl(X^{\gamma,\eta}_{t_1},X^{\gamma,\eta}_{t},X^{\gamma,\eta}_{t_2}
\bigr) >\frac{\Delta}{2}, E_\eta\cap C_\kappa \biggr)+\P
\bigl(E_\eta^c\bigr)+\P\bigl(C^c_\kappa
\bigr)+\P\bigl(E_\xi^c\bigr)
\\
&&\quad\leq\frac{\theta}{4}+\frac{\theta}{4}+\frac{\theta}{3}<\theta.\hspace*{282pt}\qed
\end{eqnarray*}
\noqed\end{pf}
\begin{pf*}{Proof of Theorem \ref{tm}}
Consider the equations (\ref{eqv}) and (\ref{eqx}) with zero initial
conditions driven by an $\alpha$-stable
L\'evy process $l^{(\alpha)}$, $\alpha\in(0,2)$. Applying the time
change $t\mapsto\frac{t}{\varepsilon^\alpha}$ and using
the self-similarity of $l^{(\alpha)}$, namely that
$\Law(\varepsilon^{1/\alpha} l^{(\alpha)}_{t/\varepsilon^\alpha},
t\geq0)=\Law(
l^{(\alpha)}_t, t\geq0)$, we obtain that for any $\varepsilon>0$
the law of the processes $(v^\varepsilon_{t/\varepsilon^\alpha
})_{t\geq0}$ and $(x^\varepsilon_{t/\varepsilon^\alpha})_{t\geq0}$ coincides with the law of
the processes $V^{1/\varepsilon^\alpha}$ and $X^{1/\varepsilon
^\alpha}$
which solve the stochastic differential equations (\ref{eqVg}) and
(\ref{eqXg})
driven by a process $L$ being a copy of $l^{(\alpha)}$, $\Law(L)=\Law(
l^{(\alpha)})$.
Then the statement of Theorem \ref{tm} follows
from Theorem \ref{thapp}.
\end{pf*}

Let us discuss the weak convergence of integrated OU processes driven
by an $\alpha$-stable L\'evy process of
small intensity for arbitrary initial conditions.

We start with the generalization of the Theorem \ref{thapp}.
Consider the system of stochastic differential equations driven by an
arbitrary L\'evy process $L$
\[
V^\gamma=v_0-\gamma\int_0^t
A V_s \,\di s+L_t,\qquad X^\gamma =x_0+\gamma
\int_0^t V_s \,\di s,
\]
$\gamma$ being a big parameter, and the initial conditions $v_0$, $x_0$
being arbitrary.
The explicit solutions are given by the formulae
%
\begin{equation}
\label{eqVsol} V^\gamma_t=v_0
\ex^{-\gamma A t}+ \int_0^t \ex^{-\gamma A(t-s)}
\,\di L_s
\end{equation}
and
%
\begin{equation}
\label{eqXsol} AX^\gamma_t=Ax_0+ v_0
\bigl(1-\ex^{-\gamma A t}\bigr) + \int_0^t
\bigl(1- \ex^{-\gamma A(t-s)} \bigr) \,\di L_s.
\end{equation}
It follows immediately from Theorem \ref{thapp} that for $x_0\in\bR$
and $v_0=0$ the processes
$(A(X^\gamma_t-x_0))_{t\geq0}$ converge in probability to $L$ in
$D([0,\infty),\bR;M_1)$.

The situation becomes a little more complicated for $v_0\neq0$. The
continuous second summand on the right-hand side of
(\ref{eqXsol}) does not converge uniformly on the intervals $[0, T]$,
$T>0$, yet has a discontinuous
point-wise limit
\[
v_0 \bigl(1-\ex^{-\gamma A t}\bigr)\to v_0
\bI_{(0,\infty)}(t),\qquad t\geq0.
\]
Thus for $v_0\neq0$, the limiting process $(L_t-v_0\bI_{\{0\}
}(t))_{t\geq0}$
is discontinuous in probability at the origin
and
the convergence in probability of finite-dimensional marginals of the process
$A(X^\gamma-x_0)-v_0$ to those of $L$ holds only on the set $(0,\infty)$.
The $M_1$-convergence to $L$ still holds on all intervals $[\delta,
T]$, $0<\delta<T$, and we obtain the following result:
for arbitrary $x_0,v_0\in\bR$
\[
\bigl(A\bigl(X^\gamma_t-x_0\bigr)-v_0
\bigr)_{t\geq0}\stackrel{\P} {\to} L \qquad\mbox{in } D\bigl((0,\infty),
\bR;M_1\bigr) \mbox{ as } \gamma\to\infty.
\]
Consequently, Theorem \ref{tm} takes the following form. Let
$l^{(\alpha)}$ be an $\alpha$-stable L\'evy process and let
$x^\varepsilon$ be the integrated OU process satisfying equations
(\ref{eqv})
and (\ref{eqx}) with arbitrary initial conditions
$x_0$, $v_0\in\bR$. Then
\[
\bigl(A\bigl(x^\varepsilon_{{t}/{\varepsilon^\alpha}}-x_0
\bigr)-v_0 \bigr)_{t\geq0} \Rightarrow l^{(\alpha)}
\qquad\mbox{in } D\bigl((0,\infty),\bR;M_1\bigr) \mbox{ as } \varepsilon\to0.
\]

Finally, we direct the reader's attention to Section 13.6.2 in Whitt
\cite{whitt02}, and Puhalskii and Whitt \cite{PuhWhi97}
for more information on the treatment of discontinuities of stochastic
processes at the origin, especially on the
so-called $M_1'$-convergence.

\section*{Acknowledgements}

I. Pavlyukevich thanks Goran Peskir for an interesting discussion about the
first passage times of (integrated)
Ornstein--Uhlenbeck processes, and Markus Riedle for various helpful
comments. The authors are indebted to the anonymous referee
for a careful reading of the manuscript, pointing out a shorter proof
of Proposition \ref{pbr}, and making
numerous valuable suggestions which significantly improved the quality
of this paper.



\printhistory


\begin{thebibliography}{18}

\bibitem{Adler90}
\begin{bbook}[mr]
\bauthor{\bsnm{Adler},~\bfnm{Robert~J.}\binits{R.J.}}
(\byear{1990}).
\btitle{An Introduction to Continuity, Extrema, and Related Topics for General
  {G}aussian Processes}.
\bseries{Institute of Mathematical Statistics Lecture Notes---Monograph Series}
\bvolume{12}.
\blocation{Hayward, CA}: \bpublisher{IMS}.
\bid{mr={1088478}}
\bptok{imsref}%
\end{bbook}
\endbibitem

\bibitem{altalibi10}
\begin{bmisc}[author]
\bauthor{\bsnm{Al-Talibi},~\bfnm{H.}\binits{H.}},
  \bauthor{\bsnm{Hilbert},~\bfnm{A.}\binits{A.}} \AND
  \bauthor{\bsnm{Kolokoltsov},~\bfnm{V.}\binits{V.}}
(\byear{2010}).
\bhowpublished{Nelson-type limit for a particular class of {L{\'e}}vy processes.
\textit{AIP Conf. Proc.}
\textbf{1232} 189}.
\bptok{imsref}%
\end{bmisc}
\endbibitem

\bibitem{Applebaum09}
\begin{barticle}[author]
\bauthor{\bsnm{Applebaum},~\bfnm{D.}\binits{D.}}
(\byear{2009}).
\btitle{Extending stochastic resonance for neuron models to general {L{\`e}}vy
  noise}.
\bjournal{IEEE Transactions on Neural Networks}
\bvolume{20}
\bpages{1993--1995}.
\bptok{imsref}%
\end{barticle}
\endbibitem

\bibitem{Bar98}
\begin{barticle}[mr]
\bauthor{\bsnm{Barndorff-Nielsen},~\bfnm{Ole~E.}\binits{O.E.}}
(\byear{1998}).
\btitle{Processes of normal inverse {G}aussian type}.
\bjournal{Finance Stoch.}
\bvolume{2}
\bpages{41--68}.
\bid{doi={10.1007/s007800050032}, issn={0949-2984}, mr={1804664}}
\bptok{imsref}%
\end{barticle}
\endbibitem

\bibitem{BarShe03}
\begin{barticle}[mr]
\bauthor{\bsnm{Barndorff-Nielsen},~\bfnm{Ole~E.}\binits{O.E.}} \AND
  \bauthor{\bsnm{Shephard},~\bfnm{Neil}\binits{N.}}
(\byear{2003}).
\btitle{Integrated {OU} processes and non-{G}aussian {OU}-based stochastic
  volatility models}.
\bjournal{Scand. J. Stat.}
\bvolume{30}
\bpages{277--295}.
\bid{doi={10.1111/1467-9469.00331}, issn={0303-6898}, mr={1983126}}
\bptok{imsref}%
\end{barticle}
\endbibitem

\bibitem{ChechkinGS02}
\begin{barticle}[author]
\bauthor{\bsnm{Chechkin},~\bfnm{A.~V.}\binits{A.V.}},
  \bauthor{\bsnm{Gonchar},~\bfnm{V.~Yu.}\binits{V.Y.}} \AND
  \bauthor{\bsnm{Szyd{\l}owski},~\bfnm{M.}\binits{M.}}
(\byear{2002}).
\btitle{Fractional kinetics for relaxation and superdiffusion in a magnetic
  field}.
\bjournal{Physics of Plasmas}
\bvolume{9}
\bpages{78--88}.
\bptok{imsref}%
\end{barticle}
\endbibitem

\bibitem{GarOlk00}
\begin{barticle}[mr]
\bauthor{\bsnm{Garbaczewski},~\bfnm{Piotr}\binits{P.}} \AND
  \bauthor{\bsnm{Olkiewicz},~\bfnm{Robert}\binits{R.}}
(\byear{2000}).
\btitle{Ornstein--{U}hlenbeck--{C}auchy process}.
\bjournal{J.~Math. Phys.}
\bvolume{41}
\bpages{6843--6860}.
\bid{doi={10.1063/1.1290054}, issn={0022-2488}, mr={1781410}}
\bptok{imsref}%
\end{barticle}
\endbibitem

\bibitem{Hesse91}
\begin{barticle}[mr]
\bauthor{\bsnm{Hesse},~\bfnm{C.~H.}\binits{C.H.}}
(\byear{1991}).
\btitle{The one-sided barrier problem for an integrated {O}rnstein--{U}hlenbeck
  process}.
\bjournal{Comm. Statist. Stochastic Models}
\bvolume{7}
\bpages{447--480}.
\bid{doi={10.1080/15326349108807200}, issn={0882-0287}, mr={1123138}}
\bptok{imsref}%
\end{barticle}
\endbibitem

\bibitem{HorLef84}
\begin{bbook}[mr]
\bauthor{\bsnm{Horsthemke},~\bfnm{Werner}\binits{W.}} \AND
  \bauthor{\bsnm{Lefever},~\bfnm{Ren{\'e}}\binits{R.}}
(\byear{1984}).
\btitle{Noise-Induced Transitions: Theory and Applications in Physics, Chemistry, and Biology}.
\bseries{Springer Series in Synergetics}
\bvolume{15}.
\blocation{Berlin}: \bpublisher{Springer}.
\bid{mr={0724433}}
\bptok{imsref}%
\end{bbook}
\endbibitem

\bibitem{Kuznetsov11}
\begin{barticle}[mr]
\bauthor{\bsnm{Kuznetsov},~\bfnm{Alexey}\binits{A.}}
(\byear{2011}).
\btitle{On extrema of stable processes}.
\bjournal{Ann. Probab.}
\bvolume{39}
\bpages{1027--1060}.
\bid{doi={10.1214/10-AOP577}, issn={0091-1798}, mr={2789582}}
\bptok{imsref}%
\end{barticle}
\endbibitem

\bibitem{Lef89}
\begin{barticle}[mr]
\bauthor{\bsnm{Lefebvre},~\bfnm{Mario}\binits{M.}}
(\byear{1989}).
\btitle{Moment generating function of a first hitting place for the integrated
  {O}rnstein--{U}hlenbeck process}.
\bjournal{Stochastic Process. Appl.}
\bvolume{32}
\bpages{281--287}.
\bid{doi={10.1016/0304-4149(89)90080-X}, issn={0304-4149}, mr={1014454}}
\bptok{imsref}%
\end{barticle}
\endbibitem

\bibitem{PuhWhi97}
\begin{barticle}[mr]
\bauthor{\bsnm{Puhalskii},~\bfnm{Anatolii~A.}\binits{A.A.}} \AND
  \bauthor{\bsnm{Whitt},~\bfnm{Ward}\binits{W.}}
(\byear{1997}).
\btitle{Functional large deviation principles for first-passage-time
  processes}.
\bjournal{Ann. Appl. Probab.}
\bvolume{7}
\bpages{362--381}.
\bid{doi={10.1214/aoap/1034625336}, issn={1050-5164}, mr={1442318}}
\bptok{imsref}%
\end{barticle}
\endbibitem

\bibitem{Sato-99}
\begin{bbook}[mr]
\bauthor{\bsnm{Sato},~\bfnm{Ken}\binits{K.}}
(\byear{1999}).
\btitle{L\'evy Processes and Infinitely Divisible Distributions}.
\bseries{Cambridge Studies in Advanced Mathematics}
\bvolume{68}.
\blocation{Cambridge}: \bpublisher{Cambridge Univ. Press}.
\bnote{Translated from the 1990 Japanese original, revised by the author}.
\bid{mr={1739520}}
\bptok{imsref}%
\end{bbook}
\endbibitem

\bibitem{Simon11}
\begin{barticle}[mr]
\bauthor{\bsnm{Simon},~\bfnm{Thomas}\binits{T.}}
(\byear{2011}).
\btitle{Hitting densities for spectrally positive stable processes}.
\bjournal{Stochastics}
\bvolume{83}
\bpages{203--214}.
\bid{doi={10.1080/17442508.2010.549232}, issn={1744-2508}, mr={2800088}}
\bptok{imsref}%
\end{barticle}
\endbibitem

\bibitem{Skorokhod56}
\begin{barticle}[mr]
\bauthor{\bsnm{Skorohod},~\bfnm{A.~V.}\binits{A.V.}}
(\byear{1956}).
\btitle{Limit theorems for stochastic processes}.
\bjournal{Theory Probab. Appl.}
\bvolume{1}
\bpages{261--290}.
\bid{mr={0084897}}
\bptok{imsref}%
\end{barticle}
\endbibitem

\bibitem{whitt02}
\begin{bbook}[mr]
\bauthor{\bsnm{Whitt},~\bfnm{Ward}\binits{W.}}
(\byear{2002}).
\btitle{Stochastic-Process Limits:
An Introduction to Stochastic-Process Limits and their Application to
  Queues}.
\bseries{Springer Series in Operations Research}.
\blocation{New York}: \bpublisher{Springer}.
\bid{mr={1876437}}
\bptok{imsref}%
\end{bbook}
\endbibitem

\bibitem{Williams01}
\begin{barticle}[mr]
\bauthor{\bsnm{Williams},~\bfnm{David R.~E.}\binits{D.R.E.}}
(\byear{2001}).
\btitle{Path-wise solutions of stochastic differential equations driven by
  {L}\'evy processes}.
\bjournal{Rev. Mat. Iberoam.}
\bvolume{17}
\bpages{295--329}.
\bid{doi={10.4171/RMI/296}, issn={0213-2230}, mr={1891200}}
\bptok{imsref}%
\end{barticle}
\endbibitem

\end{thebibliography}
\end{document}